\documentclass[a4paper,12pt]{amsart}
\usepackage{amsmath}
\usepackage{amsthm}
\usepackage{amssymb}
\usepackage[hidelinks]{hyperref}

\newcommand{\abs}[1]{\vert #1 \vert}

\DeclareMathOperator{\vol}{Vol}
\DeclareMathOperator{\tr}{trace}
\DeclareMathOperator{\di}{div}
\DeclareMathOperator{\rank}{rank}
\DeclareMathOperator{\gr}{grad}

\DeclareMathOperator{\Ric}{Ric}

\newcommand{\dif}{\mathrm{d}}
\newcommand{\CC}{\mathbb{C}}

\newcommand{\RR}{\mathbb{R}}

\newcommand{\Ss}{\mathbb{S}}

\newtheorem{pr}{Proposition}

\newtheorem{lm}{Lemma}
\theoremstyle{definition}
 
\newtheorem{re}{Remark}

\begin{document}

\title[Stability of Skyrme-related energies]{On a stability property of Skyrme-related energy functionals}

\author{Radu Slobodeanu}

\address{Department of Theoretical Physics and Mathematics, Faculty of Physics, University of Bucharest, P.O. Box Mg-11, RO--077125 Bucharest-M\u agurele, Romania \\  $\text{and}$}

\address{Institute of Mathematics, University of Neuch\^atel, 11 rue Emile Argand, 2000 Neuch\^atel, Switzerland.}

\email{radualexandru.slobodeanu@g.unibuc.ro}


\date{\today}


\begin{abstract}
\noindent
We study the stability of critical maps from (or into) spheres with respect to the symplectic Dirichlet and $\sigma_2$ energies which are the fourth power terms in Skyrme type sigma-models.
\end{abstract}

\maketitle

\section{Introduction}
Given two Riemannian manifolds $(M, g)$ and $(N,h)$, the solutions  $\varphi: M \to N$ of the variational problem associated to the \textit{Dirichlet energy}
\begin{equation*}
\mathcal{E}(\varphi) =
\frac{1}{2} \int_M \abs{\dif \varphi}^2 \nu_{g}.
\end{equation*}
are called \textit{harmonic maps}. Derrick's \textit{scaling argument} \cite{der} implies that, in any dimension other than $m = 2$, there are no non-constant finite-energy harmonic maps defined on $\RR^m$. 
This feature is not entirely specific to the euclidean metric, as proved by Sealey \cite{sea}. If $\mathcal{E}(\varphi)$ is allowed to contain a potential term $V \circ \varphi$, then there are no non-trivial solutions if $m > 2$. A counterpart for compact domains is Xin's theorem \cite{xin} asserting that if $m > 2$, then there is no non-constant \textit{stable} harmonic map defined on the unit sphere $\Ss^m$. A mirror result by Leung \cite{leu} holds for mapping taking values into the sphere. Both facts are proved using the Lawson-Simons \textit{averaging argument} \cite{ls}.

Derrick and Xin-Leung theorems can be evaded by considering higher power energies, the price to be paid consisting in restrictions on the ellipticity of the corresponding Euler-Lagrange equations.
For instance, if we consider consider the $p$-\textit{energy} ($p>2$)
\begin{equation*}
\mathcal{E}_p(\varphi) =
\frac{1}{p} \int_M \abs{\dif \varphi}^p \nu_g ,
\end{equation*}
both restrictions are relaxed to $m > p$. Moreover, the Hopf map from $\Ss^3$ to $\Ss^2$ minimizes the $p$-energy in its homotopy class for $p\geq 4$ \cite{riv}.

Another natural choice of high power functional was introduced in the seminal paper on harmonic maps \cite{eell} as
\begin{equation*}
\mathcal{E}_{\sigma_p}(\varphi)=
\frac{1}{2} \int_{M} \abs{\wedge^p \dif \varphi}^2 \nu_g,
\end{equation*}
and was called $\sigma_p$-\textit{energy} since the integrand can be also seen as $\sigma_p(\varphi^*h)$, the $p^{th}$ elementary symmetric function of the eigenvalues of $\varphi^* h$ with respect to $g$.  The fourth power case, $\mathcal{E}_{\sigma_2}(\varphi)$, was already known as the self-interaction term of Skyrme's sigma-model \cite{sky} in nuclear physics.

Motivated by the strong coupling limit of Faddeev-Niemi model \cite{fad},  Speight and Svensson \cite{sve, svee} studied the \textit{symplectic Dirichlet energy}:
\begin{equation*} 
\mathcal{F}(\varphi) = \frac{1}{2} 
\int_M \abs{\varphi^* \Omega}^2 \nu_g,
\end{equation*}
suited for maps taking values in a symplectic manifold $(N, \Omega)$. 

While Derrick's result extends immediately to these alternative energy functionals asserting the non-existence of non-trivial  finite energy solutions in dimensions above the highest degree of derivatives appearing in the integrand, Xin-Leung restriction needs a more elaborate case-by-case analysis.
This has already been done for the Yang-Mills energy (of instantons) \cite{bourg}, for the volume functional (of immersions) \cite{ls, sim}, for the $p$-energy \cite{leu, take} and for the $L^2$ norm of the pullback metric \cite{kaw}. In this short note we complete the picture for the fourth power energies by proving analogue results for $\mathcal{E}_{\sigma_2}$ and $\mathcal{F}$, and by pointing out their global counterpart. This allows us to derive stability properties also for the case when we couple each of these two functionals with the Dirichlet energy, as it is usually done in the original sigma-models.

\medskip
\begin{small}
Throughout the paper, manifolds, metrics, and maps are assumed to be smooth. On a  connected Riemannian manifold $(M,g)$ with Levi-Civita connection $\nabla$, we use the following sign conventions for the curvature tensor field $R(X,Y)Z=\nabla_X\nabla_Y Z-\nabla_Y \nabla_X Z-\nabla_{[X,Y]}Z$, and $\Delta f = \tr \nabla \dif f$ for the Laplacian on functions. 
\end{small} 

\section{Symplectic Dirichlet stability on spheres}

Let $(M,g)$ and $(N,J,h)$ be Riemannian manifolds, the second being endowed with an almost K\"ahler structure with the fundamental 2-form $\Omega(\cdot,\cdot)=h(\cdot,J\cdot)$. A map $\varphi:M \to N$ is $\mathcal{F}$-critical if the first variation of $\mathcal{F}$ at $\varphi$ vanishes, and this is proved (\cite{sve}) to be equivalent with the Euler-Lagrange equations
\begin{equation}\label{fh}
\dif \varphi\left((\delta \varphi^* \Omega)^{\sharp}\right)=0.
\end{equation}
A vacuum solution (i.e. $\varphi^* \Omega=0$) is called \textit{isotropic}. A critical map is moreover a local minimizer (stable critical point) if the second variation of the energy (the Hessian) evaluated at this map is positive definite.
For any $v \in \Gamma(\varphi^{-1}TN)$, and any $\mathcal{F}$-critical map $\varphi$, the Hessian of $\mathcal{F}$ can be calculated as (\cite{sve})
\begin{equation}\label{fhess}
\mathrm{Hess}_{\varphi}^{\mathcal{F}}(v,v)= \int _M \{\abs{\dif (\varphi^* \imath _v \Omega)}^{2} + \Omega(v, \nabla^{\varphi}_{Z_\varphi} v)\} \nu_g,
\end{equation}
where $Z_\varphi=(\delta \varphi^* \Omega)^{\sharp}$. In particular, if $\delta \varphi^* \Omega=0$, then we see that $\varphi$ is stable (it actually minimizes $\mathcal{F}$ in its homotopy class \cite{svee}). 

\begin{lm}[\cite{slobo}]\label{magic}
Let $\varphi:(M,g) \to (N,h)$ be a mapping between Riemannian manifolds. Then for any $X,Y,Z \in \Gamma(TM)$ we have
\begin{equation*}
\left( \nabla_X \varphi^* h \right)(Y, Z)= h(\nabla \dif \varphi(X, Y), \dif \varphi(Z)) + h(\dif \varphi(Y), \nabla \dif \varphi(X,Z)).
\end{equation*}
\end{lm}

\begin{re}[Averaging argument]\label{av} The method introduced in \cite{ls} in order to find necessary conditions for stability on/into spheres  consists in averaging the second variation of the respective energy functional on a particular family of gradient conformal vector fields.
Let $(a_\alpha)_{\alpha=1,...,m+1}$ be an orthonormal basis in $\RR^{m+1}$. Define $f_\alpha : \Ss^{m} \to \RR$, $f_\alpha(x)=\langle a_\alpha, x\rangle$ and take $\gr f_\alpha \in \Gamma(T\Ss^{m})$.
We have $(\gr f_\alpha)_ {x} = a_\alpha - f_\alpha(x) x$, $\abs{\gr f_\alpha}^2=1 - f_\alpha^2$ and
\begin{equation}\label{nabgr}
\nabla_X \gr f_\alpha=-f_\alpha X \qquad (X \in \Gamma(T\Ss^{m})),
\end{equation}
where $\nabla$ is the Levi-Civita connection of the canonical metric $g$ on $\Ss^{m}$. In particular, $f_\alpha$ are eigenfunctions of the Laplace operator 
corresponding to the first non-zero eigenvalue: $\Delta f_\alpha=-m f_\alpha$. It is immediate to see that $\sum_\alpha f_\alpha^2=1$ (so $\sum_\alpha f_\alpha \gr f_\alpha=0$), and that, for any $X \in \Gamma(T\Ss^{m})$, $X=\sum_\alpha g(X, \gr f_\alpha)\gr f_\alpha$. 
\end{re}

\begin{pr}
If $m > 4$ there is no non-isotropic stable $\mathcal{F}$-critical map from $\Ss^m$ to any almost K\"ahler manifold.
\end{pr}

\begin{proof}
Let $\varphi:\mathbb{S}^{m} \to (N, \Omega)$ be a smooth $\mathcal{F}$-critical map and $v_\alpha=\dif \varphi(\gr f_\alpha) \in \Gamma(\varphi^{-1}TN)$, $\alpha=1,...,m+1$ be defined using Remark \ref{av}.

Observe that $\nabla^{\varphi}_{V} \dif \varphi(X) = \dif \varphi([V,X])$, for any $V \in \ker \dif \varphi$ and any $X\in\Gamma(TM)$. Since $\varphi$ is $\mathcal{F}$-critical, $Z_\varphi \in \ker \dif \varphi$ and we have
\begin{equation}\label{h1}
\begin{split}
\Omega(v_\alpha, \nabla^{\varphi}_{Z_{\varphi}} v_\alpha)
&= - \varphi^* \Omega([Z_{\varphi}, \gr f_\alpha], \gr f_\alpha)\\
&=\varphi^* \Omega \left(\nabla_{\gr f_\alpha}Z_{\varphi}, \gr f_\alpha\right)\\
&=-\left(\nabla_{\gr f_\alpha} \varphi^* \Omega\right) \left(Z_{\varphi}, \gr f_\alpha\right),
\end{split}
\end{equation}
so by summing over $\alpha$ we obtain 
\begin{equation}\label{h1S}
\sum_\alpha
\Omega(v_\alpha, \nabla^{\varphi}_{Z_{\varphi}} v_\alpha)
= -\abs{\delta \varphi^* \Omega}^2.
\end{equation}

Using Lemma \ref{magic} and Remark \ref{av}, we obtain
\begin{equation}
\dif (\varphi^* \imath _{v_\alpha} \Omega)(X, Y)=
\left(\nabla_{\gr f_\alpha} \varphi^* \Omega\right)(X,Y)
-2f_\alpha \, \varphi^* \Omega(X,Y),
\end{equation}
so by taking the norm and summing over $\alpha$, 
\begin{equation}\label{h2S}
\sum_\alpha\abs{\dif (\varphi^* \imath _{v_\alpha} \Omega)}^{2}=\abs{\nabla \varphi^* \Omega}^2 + 4 \abs{\varphi^* \Omega}^2.
\end{equation}
Combining \eqref{h1S} and \eqref{h2S} we see that calculating the trace of the Hessian requires the following Weitzenb\"ock formula for $p$-forms
(see \cite[(1.32)]{eel} and references therein) 
\begin{equation*}
-\tfrac{1}{2}\Delta\abs{\sigma}^2 
= \langle \Delta \sigma, \sigma\rangle
 - \abs{\nabla \sigma}^2- \langle S(\sigma), \sigma \rangle ,
\end{equation*}
which by integration over a compact manifold without boundary gives:
\begin{equation}\label{weitz}
\int_M \abs{\dif \sigma}^2 + \abs{\delta \sigma}^2 - \abs{\nabla \sigma}^2- \langle S(\sigma), \sigma \rangle =0.
\end{equation}
If $\sigma \in \Lambda^2 (M)$, then the curvature operator $S$ acts as follows
$$
S(\sigma)(X_1, X_2)=\sigma(\Ric X_1, X_2)+ \sigma(X_1, \Ric X_2)
+ \sum_s \sigma(e_s, R(X_1, X_2)e_s).
$$
In particular, for 2-forms on $\Ss^m$ we simply have
$$S(\sigma)(X_1, X_2)=(2m-4)\sigma(X_1, X_2).$$

Applying \eqref{weitz} for the closed 2-form $\varphi^* \Omega$ on $\Ss^m$ we obtain
\begin{equation*}
\begin{split}
\sum_\alpha \mathrm{Hess}_{\varphi}^{\mathcal{F}}(v_\alpha, v_\alpha)&=\int_{\Ss^m} \big\{-\abs{\delta \varphi^* \Omega}^2+\abs{\nabla \varphi^* \Omega}^2 + 4 \abs{\varphi^* \Omega}^2\big\}\nu_{can}\\
&=2(4-m) \int_{\Ss^m}\abs{\varphi^* \Omega}^2 \nu_{can}
\end{split}
\end{equation*}
and the conclusion follows.
\end{proof}

\noindent This generalizes \cite[Prop. 3.4]{slobo} in the case of Hopf maps. Recall that \cite{svee} the Hopf map $\varphi : \Ss^3 \to \CC P^1$ minimizes $\mathcal{F}$ in its homotopy class.

\subsection{Full Faddeev-Niemi model} Let us now turn attention to the coupled energy
$$
\mathcal{E}(\varphi)+ \kappa \mathcal{F}(\varphi),
$$
where $\kappa$ is a positive coupling constant. A mapping $\varphi$ will be a critical point for this action if and only if:
\begin{equation}\label{fhfull}
\tau(\varphi)-\kappa J\dif \varphi\left((\delta \varphi^* \Omega)^{\sharp}\right)=0,
\end{equation}
where $\tau(\varphi)=\tr \nabla \dif \varphi$ is the \textit{tension field} of $\varphi$. Even if the the Hessian of a coupled energy is still a linear combination of the two individual Hessians, combining the averaging arguments requires caution, since in the computation of $\sum_\alpha \mathrm{Hess}_{\varphi}^{\mathcal{E}, \mathcal{F}}(v_\alpha, v_\alpha)$ we employed again the (individual) Euler-Lagrange equations. So by carefully redoing the same steps for the full energy and using this time \eqref{fhfull}, we obtain
\begin{equation*}
\begin{split}
\sum_\alpha \mathrm{Hess}_{\varphi}^{\mathcal{E}+\kappa\mathcal{F}}(v_\alpha, v_\alpha)
&=\int_{\Ss^m}\left\{(2-m)\abs{\dif \varphi}^2 + 2\kappa(4-m)\abs{\varphi^* \Omega}^2 \right\}\nu_{can}. 
\end{split}
\end{equation*}
In particular, if $m\geq 4$, then there is no non-constant stable $(\mathcal{E}+\kappa\mathcal{F})$-critical map from $\Ss^m$ to any almost K\"ahler manifold. If $m=3$,  a necessary condition for a non-isotropic $(\mathcal{E}+\kappa\mathcal{F})$-critical map $\varphi:\Ss^3 \to N^2$ to be stable is
$$
\kappa\geq \frac{\int_{\Ss^3}\abs{\dif \varphi}^2 \nu_{can}}{2\int_{\Ss^3}\abs{\varphi^* \Omega}^2\nu_{can}}
$$
For the Hopf map $\Ss^3 \to  \CC P^1 \cong \Ss^2(\tfrac{1}{2})$ this reads $\kappa \geq 1$ and it is also a sufficient condition, as proved in \cite{sve}.

\section{$\sigma_2$-Stability on spheres}
For any map $\varphi: (M^m,g) \to (N^n,h)$ between Riemannian manifolds of dimensions $m,n \geq 2$, we denote by $\sigma_2(\varphi^* h)=\sum_{i<j}\lambda_i^2\lambda_j^2$ and we call $\sigma_2$-\textit{energy} the action functional $\mathcal{E}_{\sigma_2}(\varphi)=\tfrac{1}{2}\int_M \sigma_2(\varphi^* h) \nu_g$, where $\lambda_i^2$ are the eigenvalues of $\varphi^* h$ with respect to $g$. The corresponding Euler-Lagrange equations are  (\cite{cri}, cf. also \cite{slo}) 
$$
\tr \nabla(\abs{\dif \varphi}^2 \dif \varphi-\dif \varphi \circ \mathfrak{C}_\varphi)=0,
$$
where $\mathfrak{C}_\varphi=\dif \varphi^t \circ \dif \varphi$, the (1,1)-"dual" of $\varphi^* h$, is called the \textit{Cauchy-Green tensor}. 
The second variation formula is given below. Here we shall investigate its behaviour for mappings defined on spheres and we expect to recover the result of the previous section in this case (cf. \cite{sachs} for the identity map). Since $2\sigma_2(\varphi^* h)=\abs{\dif \varphi}^4 - \abs{\varphi^* h}^2$ in order to prove this we would be tempted to simply combine the result in \cite{leu, take}
 \begin{equation*}
\sum_\alpha \mathrm{Hess}_{\varphi}^{\mathcal{E}_4}(v_\alpha, v_\alpha)=(4-n) \abs{\dif \varphi}^4
\end{equation*}
with the corresponding result in \cite{kaw} for $\mathcal{G}(\varphi)=\tfrac{1}{4}\int_M \abs{\varphi^* h}^2 \nu_g$
\begin{equation*}
\sum_\alpha \mathrm{Hess}_{\varphi}^{\mathcal{G}}(v_\alpha, v_\alpha)=(4-n) \abs{\varphi^* h}^2.
\end{equation*}
But, as already mentioned, in the derivation of these "trace" formulae the individual Euler-Lagrange equations have been employed again, so we need to identify the respective terms in order to see that this approach actually gives us the expected result. For the convenience of the reader we present the main lines of the complete proof.

\begin{lm}[The second $\sigma_2$-variation \cite{cri}]\label{sigma_2hess}
The second variation of the $\sigma_2$-energy along  $v\in\Gamma(\varphi^{-1}TN)$ evaluated on a $\sigma_2$-critical map $\varphi$ is  
\begin{equation}\label{sigma_2hes}
\begin{split}
\mathrm{Hess}_{\varphi}^{\mathcal{E}_{\sigma_2}}(v, v)&=\int_{M} \big\{ 2(\di^\varphi v)^2 + \abs{\dif \varphi}^2\left(\abs{\nabla^\varphi v}^2 - \Ric^\varphi(v,v)\right)\big\}\nu_g\\
&-\int_{M} \big\{\tfrac{1}{2}\abs{H_v}^2 + \sum_i\lambda_i^2\left(\abs{\nabla_{e_i}^\varphi v}^2 - \langle R^N(v,\dif \varphi(e_i))\dif \varphi(e_i), v \rangle \right)\big\}\nu_g\\
\end{split}
\end{equation}
where $\{e_i\}_{i=1,...,m}$ is a (local) orthonormal frame of eigenvectors of $\varphi^* h$ on $M$, $H_v(X,Y)=h(\nabla_{X}^\varphi v, \dif \varphi(Y))+h(\nabla_{Y}^\varphi v, \dif \varphi(X))$, for any $X$ and $Y$ tangent vectors to $M$, 
$\di^\varphi v=\tfrac{1}{2}\tr H_v$, and $\Ric^\varphi \left(v,w \right)=$ \\ $\sum_j h\left( R^N(v,\dif \varphi(e_j))\dif \varphi(e_j), w \right)$.
\end{lm}

\noindent As for the symplectic Dirichlet energy, we need a Weitzenb\"ock formula. 

\begin{lm}[Weitzenb\"ock type formula \cite{naka}]\label{boc}
For any smooth map $\varphi:(M,g)\to(N,h)$ the following identy holds
\begin{equation*}
\begin{split}
\frac{1}{4}\Delta\abs{\varphi^* h}^2 =&\frac{1}{2}\abs{\nabla \varphi^* h}^2
+\sum_{i}\lambda_i^2 \abs{\nabla \dif \varphi(e_i, \cdot)}^2 \\
&+\sum_{i}\lambda_i^2 \left[h\left(\dif \varphi(\Ric^M e_i), \dif \varphi(e_i)\right) - \Ric^\varphi \left(\dif \varphi(e_i),\dif \varphi(e_i)\right)\right]\\
&+\di(\mathfrak{C}_\varphi(\dif \varphi^t(\tau(\varphi)))-h(\tau(\varphi), \tr \nabla (\dif \varphi \circ \mathfrak{C}_\varphi)),
\end{split}
\end{equation*}
where $\{e_i\}_{i=1,...,m}$ is a (local) orthonormal frame of eigenvectors for $\varphi^* h$ with respect to $g$, corresponding to the eigenvalues $\{\lambda_i^2\}_{i=1,...,m}$.
\end{lm}

This identity is obtained by direct computation and not by deriving it from the general Weitzenb\"ock formula \cite[(1.34)]{eel} applied to $\dif \varphi \circ \mathfrak{C}_\varphi$. By combining it with the Weitzenb\"ock formula used in the regularity theory of $p$-harmonic maps ($p=4$) \cite{xinp} we can obtain a Weitzenb\"ock formula suited for $\sigma_2$-critical maps.

We are now ready to prove the following stability property of $\sigma_2$-energy.

\begin{pr}
If $m > 4$, then there is no stable $\sigma_2$-critical map of rank $ \geq 2$ from $\Ss^m$ to any Riemannian manifold $N^n$ $(n\geq 2)$.
\end{pr}

\begin{proof}
Let $\varphi:\mathbb{S}^{m} \to N$ be a smooth $\sigma_2$-critical map and $v_\alpha=\dif \varphi(\gr f_\alpha)$, $\alpha=1,...,m+1$ be defined in Remark \ref{av}. For the first two terms in \eqref{sigma_2hes}, a computation corresponding to the 4-harmonic case yields
\begin{equation*}
\begin{split}
&\sum_\alpha 2(\di^\varphi v_\alpha)^2 + \abs{\dif \varphi}^2\left(\abs{\nabla^\varphi v_\alpha}^2 - \Ric^\varphi(v_\alpha,v_\alpha)\right)\\
&=(4-m)\abs{\dif \varphi}^4
+h\left(\tau(\varphi), \abs{\dif \varphi}^2\tau(\varphi)+\dif \varphi(\gr \abs{\dif \varphi}^2)\right)+\di(\dots).
\end{split}
\end{equation*}
For the third term in \eqref{sigma_2hes}, by using Lemma \ref{magic} we obtain 
\begin{equation*}
\sum_\alpha \tfrac{1}{2}\abs{H_{v_\alpha}}^2=\tfrac{1}{2}\abs{\nabla \varphi^* h}^2 + 2\sum_i \lambda_i^4.
\end{equation*}
Finally, we directly check that
$$
\sum_{i}\lambda_i^2 \abs{\nabla \dif \varphi(e_i, \cdot)}^2=
\sum_{i, \alpha}\lambda_i^2 \abs{\nabla_{e_i}^{\varphi} \dif \varphi(\gr f_\alpha)}^2-
\sum_i \lambda_i^4,
$$
which, combined with the Weitzenb\"ock formula (Lemma \ref{boc}), yields
\begin{equation*}
\begin{split}
&\sum_{i,\alpha}\lambda_i^2\left[\abs{\nabla_{e_i}^\varphi v_\alpha}^2 - h\left( R^N(v_\alpha,\dif \varphi(e_i))\dif \varphi(e_i), v_\alpha \right) \right]\\
&= -\tfrac{1}{2}\abs{\nabla \varphi^* h}^2 +(2-m)\sum_i \lambda_i^4 + h\left(\tau(\varphi), \tr \nabla(\dif \varphi \circ \mathfrak{C}_\varphi)\right)+\di(\dots).
\end{split}
\end{equation*}

Inserting all in the Hessian formula \eqref{sigma_2hes} and noticing that the $\sigma_2$-Euler-Lagrange equations satisfied by $\varphi$ assure the cancellation of $h(\tau(\varphi), \dots)$, we obtain
\begin{equation*}
\begin{split}
\sum_\alpha \mathrm{Hess}_{\varphi}^{\mathcal{E}_{\sigma_2}}(v_\alpha, v_\alpha)=2(4-m) \int_{\Ss^m}\sigma_2(\varphi^*h) \nu_{can}
\end{split}
\end{equation*}
and, since $\sigma_2(\varphi^*h)\geq 0$ with equality iff $\rank \dif \varphi_x <2$ for all $x \in \Ss^m$, the conclusion follows.
\end{proof}

Starting from Equation \eqref{sigma_2hes} and applying the averaging argument (Remark \ref{av}) with $v_\alpha=(\gr f_\alpha)\circ \varphi$ yields
\begin{lm}[Stability inequality]For any stable $\sigma_2$-critical map $\varphi:M \to \Ss^n$ the folowing inequality holds
\begin{equation*}
\int_M \big\{(n-1)\left(\abs{\gr f}^2\abs{\dif \varphi}^2 -\abs{\dif \varphi (\gr f)}^2 \right) + 2(4-n)f^2 \sigma_2(\varphi^* h)\big\}\nu_g \geq 0,
\end{equation*}
where $f$ is a smooth function with compact support on $M$.
\end{lm}

Letting $f=1$ gives the non-existence result analogous to \cite{le, take, xinp}, 

\begin{pr}
If $n > 4$, then there is no stable $\sigma_2$-critical map of rank $ \geq 2$ from any compact Riemannian manifold $M^m$  $(m \geq 2)$ into $\Ss^n$.
\end{pr}

\subsection{Full Skyrme model} Let us consider the coupled energy
$$
\mathcal{E}(\varphi)+ \kappa \mathcal{E}_{\sigma_2}(\varphi),
$$
where $\kappa$ is a positive coupling constant. With the same argument as in the previous section, if $m\geq 4$, then there is no non-constant stable $(\mathcal{E}+\kappa\mathcal{E}_{\sigma_2})$-critical map from $\Ss^m$ to any Riemannian manifold. If $m=3$,  a necessary condition for a non-constant $(\mathcal{E}+\kappa\mathcal{E}_{\sigma_2})$-critical map $\varphi:\Ss^3 \to N$ to be stable is
$$
\kappa\geq \frac{\int_{\Ss^3}\abs{\dif \varphi}^2 \nu_{can}}{2\int_{\Ss^3}\sigma_2(\varphi^* h)\nu_{can}}
$$
For the identity map of $\Ss^3$ (of unit radius) this reads $\kappa \geq \tfrac{1}{2}$ and one knows that it is also a sufficient condition \cite{los, man}; see also \cite{slo}.

\section{Infima in homotopy classes}

In this section we point out a global analogue of the results in the previous sections.

\begin{lm} \label{ineq} $(i)$ \ Let $\varphi: (M, g) \to (N^n, \Omega, h)$ be a (smooth) map into an almost K\"ahler manifold with fundamental $2$-form $\Omega$. Then
$$
\abs{\varphi^* \Omega}^2 \leq \sigma_2(\varphi^*h) ,
$$
where the equality is reached if and only if $n=2$.

\noindent $(ii)$ \ Let $\varphi: (M, g) \to (N^n,h)$ be a (smooth) map between Riemannian manifolds. Then
$$
\sigma_2(\varphi^*h) \leq \tfrac{n-1}{2n}\abs{\dif \varphi}^4,
$$
where the equality is reached if and only if $\varphi$ is semi-conformal (i.e., the eigenvalues of $\varphi^*h$ are all equal).
\end{lm}

\begin{proof}
$(i)$ Let $\{e_i\}_{i=1,...,m}$ be a (local) orthonormal frame of eigenvectors of $\varphi^*h$. Applying Cauchy inequality $\abs{\varphi^* \Omega}^2 =\sum_{i<j}h(\dif \varphi(e_i), J\dif \varphi(e_j))^2$ $\leq
\sum_{i<j}\abs{\dif \varphi(e_i)}^2\abs{\dif \varphi(e_j)}^2=\sum_{i<j}\lambda_i^2 \lambda_j^2
= \sigma_2(\varphi^*h)$, gives us the result. 

$(ii)$ This is one of the Newton's inequalities.
\end{proof}

Since, by \cite{wei, whi}, the infimum of the 4-energy in each homotopy class of mappings from (or into) a sphere of dimension greater than 4 is zero, Lemma \ref{ineq} implies the following

\begin{pr}
If $m > 4$, then the infimum of the symplectic Dirichlet energy and of $\sigma_2$-energy in any homotopy class of maps $\Ss^m \to N$ or $M \to \Ss^m$ ($M,N$ compact) is zero.
\end{pr}

We include here an elementary proof for the first part of the result, analogous to the Dirichlet energy case \cite{eel, eell}. If a homotopy class of mappings $M \to \Ss^m$ contains a Riemannian submersion, the proof of the second statement is similar.

\begin{proof}
Let $m \geq 5$, $c>0$ and $\phi_c$ be defined (by suspension) between charts of $\Ss^m$ as
$$
(\cos s, \sin s \cdot z) \mapsto (\cos \alpha(s), \sin \alpha(s) \cdot z); \quad \alpha(s)=2\arctan(c \tan(\tfrac{s}{2})),
$$
where $0\leq s < \pi$ and $z \in \Ss^{m-1}$. Notice that this defines indeed a smooth map $\phi_c:\Ss^m \to \Ss^m$ (regular at the poles), which has topological degree 1, and is conformal of dilation
$$
\lambda^2=\frac{c(1+\tan^2(\tfrac{s}{2}))}{1+c^2\tan^2(\tfrac{s}{2})},
$$
where we considered $\Ss^m$ endowed with the canonical metric (we can show as in \cite{eel} that  the statement we wish to prove is independent of the choices of metrics on the domain or codomain).  Since $m \geq 5$,
\begin{equation*}
\begin{split}
\mathcal{E}_4(\phi_c) &= \frac{\vol(\Ss^{m-1})}{4}\int_{0}^{\pi} \lambda^4 \sin^{m-1}s \, \dif s\\
&=
\frac{\vol(\Ss^{m-1})}{4}\int_{0}^{\pi}\left(\frac{2c\tan(\tfrac{s}{2})}{1+c^2\tan^2(\tfrac{s}{2})}\right)^4\sin^{m-5}s \, \dif s\\
&\leq
\frac{\vol(\Ss^{m-1})}{4}\int_{0}^{\pi}\left(\frac{2c\tan(\tfrac{s}{2})}{1+c^2\tan^2(\tfrac{s}{2})}\right)^4 \dif s=
\frac{\vol(\Ss^{m-1})\pi c(c^2+4c+1)}{4(c+1)^4},
\end{split}
\end{equation*}
so $\lim_{c\to 0}\mathcal{E}_4(\phi_c)=0$.

Now let $N$ be a compact manifold and $\varphi:\Ss^m \to N$. Then $\varphi_c = \varphi \circ \phi_c$ is homotopic with $\varphi$. By an elementary (algebraic) property of Hilbert-Schmidt norm,
$$
\abs{\dif \varphi_c}^4 \leq \abs{\dif \varphi}^4\abs{\dif \phi_c}^4,
$$
so we can conclude that $\lim_{c\to 0}\mathcal{E}_4(\varphi_c) = 0$. Combining with Lemma \ref{ineq} allows us to conclude that the infimum of the symplectic Dirichlet energy and of $\sigma_2$-energy in the homotopy class of $\Ss^m \to N$ is zero. 
\end{proof}

\section{Final remarks}
In this note we restricted to the $\sigma_2$ and symplectic Dirichlet energies since they correspond to Lagrangians which are at most
quadratic in first time derivatives (a requirement for any field theory with standard Hamiltonian). Nevertheless the results here should have straightforward extensions to other higher power functionals as $\sigma_p$.
Also we discussed only the sphere case, but we had in mind that the same phenomena should occur not only on product of spheres but also on other symmetric spaces as it was proved for ($p$-)harmonic maps (\cite{ohn, mont, wei}). Most notably it would be interesting to find the stability properties of the symplectic Dirichlet energy for maps defined on a complex projective space. Direct application of the averaging argument suited to $\CC P^m$ (\cite{oh}) has failed to provide us with an effective criterion of stability. In this case the use of a different basis of vectors for the averaged Hessian seems to impose (most probably symplectic vectors).
 

\end{document}